\documentclass[11pt]{amsart}

\usepackage[T1]{fontenc}
\usepackage[utf8]{inputenc}
\usepackage{amsmath,amssymb,amsfonts,amsthm,mathtools}
\usepackage{mathrsfs}
\usepackage{hyperref}
\usepackage{footnote}
\usepackage{multicol}
\usepackage{booktabs}
\usepackage[flushleft]{threeparttable}
\usepackage[font=small,labelfont=bf]{caption}

\hypersetup{
    colorlinks=true,
    linkcolor=blue,
    citecolor=blue,
    urlcolor=blue
}

\theoremstyle{plain}
\newtheorem{theorem}{Theorem}[section]
\newtheorem{lemma}[theorem]{Lemma}
\newtheorem{proposition}[theorem]{Proposition}
\newtheorem{corollary}[theorem]{Corollary}
\newtheorem{example}[theorem]{Example}
\newtheorem{remark}[theorem]{Remark}

\title[Carleson Measures and Kernel Estimates]
{Carleson-Type Measures and Kernel Estimates for Potential--Harmonic Weighted Bergman Spaces on the Unit Ball}

\author{Nihat G\"okhan G\"og\"u\c{s}}
\address{Sabancı University, Faculty of Engineering and Natural Sciences, Istanbul, Turkey}
\email{nggogus@sabanciuniv.edu}

\author{Sinem Yelda S\"onmez}
\address{Altınbaş University, Faculty of Engineering and Architecture, Istanbul, Turkey}
\address{University of Helsinki, Department of Mathematics and Statistics, Helsinki, Finland}
\email{sinem.sonmez@altinbas.edu.tr}
\email{sinem.sonmez@helsinki.fi}

\begin{document}

\begin{abstract}
In this paper, weighted Bergman spaces on the unit ball in $\mathbb{C}^n$ are investigated. A characterization of the Carleson embeddings is established. Pointwise and norm estimates on the reproducing kernel function of weighted Bergman spaces on the unit ball are proved.
\end{abstract}

\subjclass[2020]{30H20, 32A36, 46E15}
\keywords{Bergman spaces, Carleson measures, kernel estimates}

\maketitle

\noindent\textit{Accepted for publication in Hacettepe Journal of Mathematics and Statistics on March 15, 2026.}
\medskip

\section{Introduction and Preliminaries}

Weighted Bergman spaces on the unit ball $\mathbb{B}\subset\mathbb{C}^n$ form a fundamental bridge between several complex variables, operator theory, and modern function space analysis. They constitute a natural setting for the study of boundedness, compactness, and spectral properties of Toeplitz, composition, and integral operators, for sampling and interpolation phenomena, and for precise kernel asymptotics connecting analytic and geometric structures. In recent years, the theory has undergone a notable expansion driven by the introduction of highly nontrivial weight classes, including regular and doubling weights, rapidly increasing weights, and nonradial potential type weights, which provide refined control over the boundary geometry and internal singularities of analytic functions~\cite{KeheZhu}, \cite{PelaezRattya}, \cite{DuLiLiuShi}, \cite{PauZhao}.

Motivated by these developments, the present paper focuses on a large and flexible class of weighted Bergman spaces $A^p_{\omega}(\mathbb{B})$, where the weight $\omega$ is generated by an interior potential term and a boundary harmonic component. More precisely, we consider
\[
\omega_{\mu,q,s,\nu}(z)
= (1-|z|^2)^q\,U_{\mu,s}(z) + P\nu(z),
\qquad
U_{\mu,s}(z)=\int_{\mathbb{B}} G^s(z,w)\,d\mu(w),
\]
where $\mu$ and $\nu$ are finite positive measures on $\mathbb{B}$ and $\partial\mathbb{B}$, respectively, and $G$ denotes the Green's potential kernel. This definition encompasses several well-studied models, such as the purely harmonic case when $\mu=0$, the purely potential case when $\nu=0$, and perturbed radial weights for suitable $(\mu,\nu)$. It also allows one to treat nonradial, measure driven perturbations within a unified analytic framework. A crucial feature of this class is its local comparability on Bergman balls, which permits a precise geometric control essential for testing function constructions and kernel estimates.

Our results provide two central contributions. 
First, we characterize $\tilde{p}$–Carleson measures for $A^p_\omega(\mathbb{B})$ by a local testing condition involving Bergman balls. This characterization gives a transparent and easily verifiable criterion that captures the embedding geometry induced by $\omega_{\mu,q,s,\nu}$. Second, we obtain two sided norm bounds and sharp local and global pointwise estimates for the reproducing kernel. The local estimates yield quantitative equivalence of kernel values on small Bergman balls, while the global ones, proved via an adaptation of Berndtsson’s $\bar\partial$ method, extend the known off diagonal kernel control to this highly nonradial class of weights. Together, these results unify and extend existing Carleson type and kernel frameworks from the planar setting to the unit ball in higher dimensions.

The question of characterizing Carleson measures for weighted Bergman spaces has a long history, originating from Luecking’s testing techniques~\cite{Luecking} and continuing through deep developments by Hastings, Oleinik, and Pavlov, among others. More recent progress has been achieved for regular and rapidly increasing weights~\cite{DuLiWulan}, for geometric and exponential weights defined via intrinsic metrics~\cite{ChoLeePark, ChoPark}, and for nonradial measures using potential theoretic constructions~\cite{PelaezRattya, PauZhao}. Our approach builds upon and connects these threads. It complements the regular or doubling setting established by Du, Li, Liu, and Shi~\cite{DuLiLiuShi} by allowing potential and harmonic perturbations beyond radial symmetry and provides kernel technology compatible with the geometric approaches used in exponential weight theories. The characterization we obtain serves as a unified framework encompassing both Carleson embeddings and kernel estimates under a single structural assumption, namely the local comparability of the weight.

The present work therefore contributes to the ongoing shift in Bergman space theory from purely radial models toward geometrically and analytically flexible settings. The mixed potential and harmonic weights considered here enable one to model both interior and boundary phenomena and to analyze operator theoretic problems, such as Toeplitz and Cesàro operators or Schatten class criteria, in a substantially broader context. Moreover, the kernel estimates we derive are expected to play a role in subsequent studies on sampling, reverse Carleson measures, and integral operator bounds in several complex variables.

The paper is organized as follows. Section 1 recalls the geometric and measure-theoretic background on $\mathbb{B}$ and introduces the basic properties of $\omega_{\mu,q,s,\nu}$, including its local comparability. Section 2 contains the proof of the Carleson embedding characterization based on testing functions adapted to $\omega$. Section 3 establishes norm and pointwise kernel estimates through $\bar\partial$ techniques and localization on Bergman balls.

Let $\mathbb{B}$ denote the unit ball in $\mathbb{C}^n$ and let $dv$ denote the Lebesgue measure on $\mathbb{C}^n$. Let $\omega$ be a strictly positive function on an open neighborhood $N$ of the boundary $\partial \mathbb{B}$ of the unit ball and let $\omega \ge 0$ on the unit ball so that $\inf_{L} \omega >0$ for every compact subset $L$ of $N$. Such a function will be called a weight function.

Firstly, we define $\omega_{\mu,q,s,\nu}$ on the unit ball $\mathbb{B}$ as a weight function, for any positive finite Borel measure $\mu$ on the unit ball, and positive finite Borel measure $\nu$ on the boundary $\partial \mathbb{B}$. Basically, the function $\omega_{\mu,q,s,\nu}$ can be considered as a sum of a weighted potential and a harmonic function. We characterize Carleson measures for these weighted Bergman spaces in the second section.

Let $0<p<\infty$ and let $H(\mathbb{B})$ denote the space of all holomorphic function on the unit ball. The weighted Bergman space $A^p_\omega(\mathbb{B})$ with a weight function $\omega$ is the space of all holomorphic functions $f\in H(\mathbb{B})$ which satisfies the following
\begin{equation*}
    \|f\|_{A^p_\omega(\mathbb{B})}^p := \int_{\mathbb{B}} \vert f(z)\vert^p \omega(z) dv(z) < \infty.
\end{equation*}

In the third section, we provide norm and pointwise estimates on the kernel function of the weighted Bergman space $A^2_{\omega_{\mu,q,s,\nu}}(\mathbb{B})$. In order to prove a global pointwise estimate on the kernel function, we utilize a $\overline{\partial}$ method of Berndtsson, which is obtained in \cite{Berndtsson}.

One can find the following motivations in \cite{KeheZhu}. For any $w\in\mathbb{B}$, the orthogonal projection from $\mathbb{C}^n$ onto the one dimensional subspace $[w]$ generated by $w$ is defined by
\begin{equation*}
    P_w(z)=\frac{\langle z,w\rangle}{\vert w\vert^2}w,
\end{equation*} for $z\in \mathbb{C}^n$. The orthogonal projection from $\mathbb{C}^n$ onto $\mathbb{C}^n \ominus [w]$ is defined by
\begin{equation*}
    Q_w(z)= z- \frac{\langle z,w\rangle}{\vert w\vert^2}w,
\end{equation*} for $z\in \mathbb{C}^n$. For any $w\in\mathbb{B}$, the involution map $\varphi_w$ is defined by 
\begin{equation*}
    \varphi_w(z)= \frac{w-P_w(z)-s_wQ_w(z)}{1-\langle z,w\rangle},
\end{equation*} for $z\in \mathbb{B}$, where $s_w=\sqrt{1-\vert w\vert^2}$.

For each $ z\in \mathbb{B}$, the involution map $\varphi_z$ satisfies the following property
\begin{equation}\label{involution}
1-\vert \varphi_z(w)\vert^2 =\frac{(1-\vert z\vert^2)(1-\vert w\vert^2)}{\vert 1-\langle w,z\rangle\vert^2 },
\end{equation} for $w\in \mathbb{B}$, \cite{KeheZhu}.

The notation $f\lesssim g$ means that there exists a positive constant $C$ so that $f\le Cg$. If $f\lesssim g$ and $g \lesssim f$, then we use the notation $f\approx g$.

\begin{lemma}\cite{Rudin}
    Let $n\ge 2$ be an integer, then there are constants $C_1$ and $C_2$ such that for all $z\in \mathbb{B}-\{0\}$
    \begin{equation*}
        C_1 (1-\vert z\vert^2)^n \vert z\vert^{-2(n-1)} \le g(z) \le C_2 (1-\vert z\vert^2)^n \vert z\vert^{-2(n-1)},
    \end{equation*} where
    \begin{equation*}
        g(z)= \frac{n+1}{2n} \int_{\vert z\vert}^1 r^{-2n+1} (1-r^2)^{n-1} dr,
    \end{equation*} which is called the (invariant) Green's function of $\mathbb{B}$.
\end{lemma}
Moreover, we utilize the definition of Green's function in \cite{Ullrich}. For $z$, and $w\in \mathbb{B}$,
\begin{equation*}
    G(z,w):=g(\varphi_z(w)).
\end{equation*}
Therefore, the following equivalence is obtained
\begin{equation}\label{GreenEqu}
    G (z,w) \approx (1-\vert \varphi_z(w)\vert)^n \vert \varphi_z(w)\vert^{-2(n-1)},
\end{equation} for $z,w\in \mathbb{B}$.

The pseudohyperbolic metric $\gamma$ on $\mathbb{B}$ is defined by $\gamma (z,w)=\vert \varphi_z(w)\vert$, for $z,w\in \mathbb{B}$. The pseudohyperbolic metric ball at $z\in \mathbb{B}$ is defined by
\begin{equation*}
    \Delta (z,r):= \{ w\in \mathbb{B} : \gamma(z,w)<r\}, \quad r>0.
\end{equation*} 
For $z,w\in \mathbb{B}$, the Bergman metric $\beta$ is defined by the following equation,
\begin{equation*}
    \beta (z,w) = \frac{1}{2} \log \frac{1+\vert \varphi_z(w)\vert}{1-\vert \varphi_z(w)\vert}.
\end{equation*}
The Bergman metric ball at $z\in \mathbb{B}$ is defined by
\begin{equation*}
D(z,r):=\{w\in \mathbb{B}: \beta (z,w)<r \}, \quad r>0.
\end{equation*}

 The relation between the pseudohyperbolic metric $\gamma$ and the Bergman metric $\beta$ is given in \cite{KeheZhu} by the following equation,
 \begin{equation*}
     \gamma (z,w)=\tanh (\beta (z,w)),
 \end{equation*} for $z,w\in \mathbb{B}$. In addition, $\Delta (z,r)$ is an ellipsoid, the precise formulas are given in \cite{Duren}, \cite{Rudin}, and \cite{KeheZhu}. For the given $z\in \mathbb{B}$ and $r\in(0,1)$,
 \begin{equation*}
     \beta:=\frac{1-r^2}{1-r^2\vert z\vert^2} z, \quad \text{and} \quad t=\frac{1-\vert z\vert^2}{1-r^2\vert z\vert^2}
 \end{equation*} are defined. Note that $\beta \in [z]$. Hence, the point $\beta$ is the center of ellipsoid $\Delta (z,r)$. The intersection of $\Delta (z,r)$ with $[z]$ is the one dimensional disk of radius $rt$, and the intersection of $\Delta (z,r)$ with the ortho-complement $[z]$ is an $(n-1)$ dimensional Euclidean ball of radius $r\sqrt{t}$.

\begin{lemma}[\cite{KeheZhu}] For any $z\in \mathbb{B}$ and $r>0$ the volume of the Bergman metric ball $D(z,r)$ is given by
\begin{equation*}
    v(D(z,r))=\frac{R^{2n} (1-\vert z\vert^2)^{n+1}}{(1-R^2\vert z\vert^2)^{n+1}},
\end{equation*} where $R=tanh(r)$. Moreover, for any positive $r$ there exist positive constants $c_r$ and $C_r$ such that
\begin{equation}\label{VolumeBergman}
    c_r(1-\vert z\vert^2)^{n+1} \le v(D(z,r)) \le C_r(1-\vert z\vert^2)^{n+1},
\end{equation} for every $z\in \mathbb{B}$.

\end{lemma}

\begin{lemma}\cite{KeheZhu}\label{Lemma1.3}
For every positive $r$ there exists a positive constant $C$ such that 
\begin{equation}\label{Equiv1}
    \frac{1}{C}\le \frac{1-\vert z\vert^2}{1-\vert w\vert^2}\le C, \qquad \text{and} \qquad \frac{1}{C}\le \frac{1-\vert w\vert^2}{\vert 1-\langle w,z\rangle\vert}\le C,
\end{equation} for every $z,w\in \mathbb{B}$ where $\beta (z,w)<r$.
Moreover, if $r$ is bounded above, then the constant $C$ may be chosen to be independent of $r$.
\end{lemma}
Additionally, it is concluded in \cite{KeheZhu} that for any positive $R$ there exists a positive constant $C$ such that
\begin{equation}\label{Equiv2}
    C^{-1} \vert 1- \langle z,a \rangle \vert \le \vert 1-\langle z,b \rangle \vert \le C \vert 1- \langle z,a \rangle \vert,
\end{equation} for every $z\in \mathbb{B}$ where $\beta (a,b)\le R$.

Let $\mu$ be a positive finite Borel measure on $\mathbb{B}$ and $s>0$. We define 
\begin{equation}\label{DefnUmus}
U_{\mu,s}(z):=\int_{\mathbb{B}} G^s(z,w)d\mu(w),
\end{equation} for $z\in \mathbb{B}$. The equivalence \eqref{GreenEqu} provides the following estimate
\begin{equation}\label{UmusGreen}
    U_{\mu,s}(z) \approx \int_{\mathbb{B}} (1-\vert \varphi_z(w)\vert^2)^{ns} \vert \varphi_z(w) \vert^{-2s(n-1)} d\mu(w), \quad z\in \mathbb{B}.
\end{equation} Therefore, the equations \eqref{involution} and \eqref{GreenEqu}, and the inequalities \eqref{Equiv1} and \eqref{Equiv2} imply that  $U_{\mu,s}(z)\approx U_{\mu,s}(a)$ for every $z\in D(a,r)$.

	The (invariant) Poisson kernel is defined by
	\begin{equation*}
		P(z,\xi):=\frac{(1-\vert z\vert^2)^n}{\vert 1-\langle z,\xi\rangle \vert^{2n}},
	\end{equation*} for $z\in \mathbb{B}$ and $\xi \in \partial \mathbb{B}$.
	The (invariant) Poisson integral of a measure $\nu$ on $\partial \mathbb{B}$ is defined
	\begin{equation*}
	P_\nu(z):=\int_{\partial \mathbb{B}} P(z,\xi)d\nu(\xi),
	\end{equation*} for $z\in \mathbb{B}$, \cite{Rudin}.

The definition of the Poisson integral and the inequalities \eqref{Equiv1} and \eqref{Equiv2} imply that  $P_\nu (z)\approx P_\nu (a)$, for every $z\in D(a,r)$. 

\par Let $\mu$ be a positive finite Borel measure on the unit ball $\mathbb{B}$, $\nu$ be a positive finite Borel measure on the boundary $\partial \mathbb{B}$. Assume that $p>0$, $s\geq 0$, $q+s>-1$, $q>-2$, and set
\begin{equation}\label{Defnomega}
\omega_{\mu,q,s,\nu} (z):= (1-\vert z\vert^2)^q U_{\mu,s}(z) + P_\nu(z),
\end{equation} for $z\in \mathbb{B}$. In this paper, we will mainly focus on weight functions defined by \eqref{Defnomega}.

\begin{proposition}\label{Prop:omegaEquiv}
    Let $\omega_{\mu,q,s,\nu}$ be defined as in \eqref{Defnomega} where $\mu$ and $\nu$ are positive finite Borel measures on $\mathbb{B}$ and $\partial \mathbb{B}$, respectively. For any $a\in \mathbb{B}$ and $0<r<1$ the following statement is satisfied:
    \begin{equation} \label{omegaEquivalence}
    \omega_{\mu,q,s,\nu} (z) \approx \omega_{\mu,q,s,\nu} (a), \quad z\in D(a,r).
\end{equation}
\end{proposition}

\begin{proof}
     The proof follows from the above explanations. 
\end{proof}

The function $\omega_{\mu,q,s,\nu}$ is positive on the unit ball, since each function in the definition of the weight function \eqref{Defnomega} is positive on the unit ball. The function $U_{\mu,s}$ is lower semi continuous on $\mathbb{B}$, which follows from the Fatou's lemma. Besides, the functions $(1-\vert z\vert^2)^q$ and $P_\nu (z)$ are continuous on $\mathbb{B}$. Hence,  $\omega_{\mu,q,s,\nu}$ is lower semi continuous on $\mathbb{B}$. Thus, $\omega_{\mu,q,s,\nu}$ attains its infimum in $K$, which is positive, for every compact subset $K\subset \mathbb{B}$. Therefore,  $\omega_{\mu,q,s,\nu}$ is a weight function on $\mathbb{B}$.

\section{Carleson Measure Characterization}

Let $\eta$ be a positive Borel measure on the unit ball. We say that $\eta$ is a $\Tilde{p}-$Carleson measure for $A^2_{\omega}(\mathbb{B})$, if the following inequality holds

\begin{equation*}
    \left( \int_{\mathbb{B}} \vert f\vert^{\Tilde{p}} d\eta \right)^{1/\Tilde{p}} \lesssim \| f\|_{A^2_{\omega}(\mathbb{B})},
\end{equation*} for every $f\in A^2_{\omega}(\mathbb{B})$.

\begin{lemma}\label{TestFunctionsCar}
	Let $t>0$ and $\omega:=\omega_{\mu,q,s,\nu}$. We define
	\begin{align}  \label{TestFuncsAPQSU}
	f_{w,t}(z)=\left ( \frac{(1-\lvert w\rvert^2)^{t+n-1}}{\omega(w)}\right
	)^{1/p}\frac{1}{( 1- \langle z,w\rangle )^{(2n+t)/p}},
	\end{align} for $z\in \mathbb{B}$ and $t+q>n+1 > 2+q$. Then, $\sup_{w\in\mathbb{B}}\|f_{w,t}\|_{A^p_{\omega_{}}}<\infty$.
 
 Furthermore; if $t+n-1>q+ns$, then the functions $f_{w,t}$ locally uniformly goes to zero as $\vert w\vert$ approaches to $1$.
\end{lemma}

\begin{proof}

 \begin{align*}
     \|f_{w,t}&\|^p  = \int_{\mathbb{B}} \frac{(1-\vert w\vert^2)^{{t+n-1}}}{\omega (w) \vert 1- \langle z,w\rangle \vert^{2n+t}} \omega(z) dv(z) \\
     & \le \frac{(1-\vert w\vert^2)^{t+n-1}}{\omega(w)}\\
     &\times \left( \int_{\mathbb{B}} \frac{(1-\vert z\vert^2)^q U_{\mu,s}(z)}{\vert 1- \langle z,w\rangle \vert^{2n+t}} dv(z) +  \int_{\mathbb{B}} \frac{P_\nu(z)}{\vert 1- \langle z,w\rangle \vert^{2n+t}} dv(z) \right) \\
     & \le \frac{(1-\vert w \vert^2)^{t+n-1-q}}{U_{\mu,s}(w)} \\
     &\times \int_{\mathbb{B}} \frac{(1-\vert z\vert^2 )^q}{\vert 1- \langle z,w\rangle \vert^{2n+t}}\left(\int_{\mathbb{B}} \frac{(1-\vert z\vert^2)^{ns} (1-\vert \zeta \vert^2)^{ns}}{\vert 1-\langle \zeta ,z\rangle\vert^{2ns}} d\mu(\zeta) \right)dv(z) \\
     & +  \frac{(1-\vert w\vert^2)^{t+n-1}}{P_\nu(w)}  \int_{\mathbb{B}} \frac{1}{\vert 1- \langle z,w\rangle \vert^{2n+t}} \left( \int_{\partial \mathbb{B}} \frac{(1-\vert z\vert^2)^n}{\vert 1-\langle z,a\rangle\vert^{2n}} dv(a) \right) dv(z). 
 \end{align*} 
 
 Let $I$ denote the first term of the above sum and let $II$ denote the second term of the above sum. The sixth statement of Theorem 3.1 in \cite{GLZ} implies that

\begin{align*}
    I & \approx \frac{((1-\vert w\vert^2)^{t+n-1-q})}{U_{\mu,s}(w)} \int_{\mathbb{B}} \frac{(1-\vert \zeta \vert^2)^{ns}}{(1-\vert w\vert^2)^{n-1+q-ns}\vert 1-\langle \zeta ,w\rangle \vert^{2ns} } d\mu(\zeta) \\
    & = \frac{(1-\vert w\vert^2)^{t}}{U_{\mu,s}(w)}  \int_{\mathbb{B}} \frac{(1-\vert \zeta \vert^2)^{ns} (1-\vert w\vert^2)^{ns} }{\vert 1-\langle \zeta ,w\rangle \vert^{2ns} } d\mu(\zeta) \\
    & \le (1-\vert w\vert^2)^{t},
\end{align*} since $t\ge 0$.
Similarly, using the sixth statement of Theorem 3.1 in \cite{GLZ} we obtain the following

\begin{align*}
    II & \approx \frac{(1-\vert w\vert^2)^{t+n-1}}{P_\nu(w)} \int_{\partial \mathbb{B}} \frac{1}{(1-\vert w\vert^2)^{t-1} \vert 1-\langle w,a\rangle\vert^{2n}} dv(a) \\
    & = \frac{1}{P_\nu(w)} \int_{\partial \mathbb{B}} \frac{(1-\vert w\vert^2)^{n} }{\vert 1-\langle w,a\rangle\vert^{2n}} dv(a) =1.
\end{align*} 
Thus, it is concluded that $\sup_{w\in \mathbb{B}} \|f_{w,t}\|^p  \lesssim 2$.

In order to prove the second statement of the lemma, it suffices to show that $f_{w,t}$ goes to zero on compact subsets of the unit ball. Let $K\subset \mathbb{B}$ be a compact subset such that $\mu(K)=C>0$, and let $\vert w\vert <r$ for every $w\in K$ and for some $0<r<1$. For any $z\in \mathbb{B}$ and any $w\in K$, we obtain the following inequality
\begin{equation*}
    \frac{1-\vert w\vert^2}{\vert 1-\langle w,z \rangle\vert^2} \le \frac{1}{(1-r\vert z\vert)^2}.
\end{equation*}
 Hence, for any $z\in \mathbb{B}$, the equation \eqref{involution} implies that
\begin{equation*}
    U_{\mu,s} (z) \gtrsim (1-\vert z\vert^2)^{ns} \frac{\mu(K)}{(1-r)^{2ns}}.
\end{equation*} Moreover, we obtain that
\begin{align*}
    \frac{(1-\lvert w\rvert^2)^{t+n-1}}{\omega(w)} & = \frac{(1-\lvert w\rvert^2)^{t+n-1}}{(1-\vert w\vert^2)^q U_{\mu,s}(w) + P_\nu(w)} \\
    &\lesssim \frac{(1-\vert w \vert^2)^{t+n-1}}{(1-\vert w\vert^2)^{q+2ns}} \\
   & = (1-\vert w\vert^2)^{t+n-1-q-ns},
\end{align*} since $P_\nu (w)\ge 0$.
Therefore, it is concluded that
\begin{align*}
     |f_{w,t}(z)|^p &= \frac{(1-\lvert w\rvert^2)^{t+n-1}}{\omega(w)} \frac{1}{\vert 1- \langle z,w\rangle \vert^{(2n+t)}} \\
     & \lesssim \frac{(1-\vert w\vert^2)^{t+n-1-q-ns}}{\vert 1- \langle z,w\rangle \vert^{(2n+t)}},
\end{align*} which completes the proof.
\end{proof}

A Carleson measure characterization for a positive finite Borel measure is established by Duren, in \cite{Duren}, for the classical Bergman spaces defined on the unit ball $A^p(\mathbb{B})$. A Carleson measure characterization for the weighted Bergman space on the Hartogs triangles is established, where the weight function is a real power of boundary distance function, in \cite{Zhang}.

\begin{theorem}\label{Thm22}
    Let $\omega:=\omega_{\mu,q,s,\nu}$, where $\mu$ and $\nu$ are positive finite Borel measures on $\mathbb{B}$ and $\partial \mathbb{B}$, respectively. Assume that $\eta$ is a positive finite Borel measure on the unit ball. Consider the following statements:
\begin{enumerate}
    \item \label{Thm2.1:1} The measure $\eta$ is a $\Tilde{p}$-Carleson measure for $A^p_\omega(\mathbb{B})$.
    \item \label{Thm2.1:2} We have
    \begin{equation*} 
        \sup_{w\in \mathbb{B}} \frac{\eta (D(w,r))}{ (\omega(w))^{\Tilde{p}/p} (1-\vert w\vert^2)^{(n+1)\Tilde{p}/p} } < \infty .
    \end{equation*} 
\end{enumerate}
    Then, we have that \eqref{Thm2.1:1} implies \eqref{Thm2.1:2} for every positive $\Tilde{p}$; and  \eqref{Thm2.1:2} implies \eqref{Thm2.1:1} for $\Tilde{p}\ge p$.
\end{theorem}

\begin{proof}
    Assume that $\eta$ is a $\Tilde{p}$-Carleson measure for $A^p_\omega(\mathbb{B})$. Let $f_{w,t}$ be the functions defined as in Lemma \ref{TestFunctionsCar}. For any $w\in \mathbb{B}$ and $0<r<1$, we have 
\begin{align*}
  \frac{\eta (D(w,r))}{ (\omega(w))^{\Tilde{p}/p}  (1-\vert w\vert^2)^{(n+1)\Tilde{p}/p} } & \lesssim  \int_{D(w,r)} \frac{(1-\vert w\vert ^2)^{(t+n-1)\Tilde{p}/p} }{(\omega(w))^{\Tilde{p}/p}\vert 1-\langle w,z\rangle \vert^{(2n+t)\Tilde{p}/p}} d\eta(z) \\
 & \le \int_{\mathbb{B}} \vert f_{w,t}(z) \vert^{\Tilde{p}} d\eta (z) \\
 & \le  \| f_{w,t} \|^{\Tilde{p}}_{A^p_\omega(\mathbb{B})}, 
\end{align*} which implies \eqref{Thm2.1:2}.

\par Suppose  that \eqref{Thm2.1:2} is satisfied. Let $f_{w,t}$ be the functions defined as in Lemma \ref{TestFunctionsCar}. It follows from the assumption that
\begin{align*}
    \eta (D(w,r)) & \lesssim (\omega(w))^{\Tilde{p}/p} (1-\vert w\vert^2)^{(n+1)\Tilde{p}/p} \\
    & \lesssim \int_{D(w,r)} (\omega(w))^{\Tilde{p}/p} (1-\vert w\vert^2)^{(n+1)\Tilde{p}/p -n-1} dv(z).
\end{align*} Hence, \cite{Luecking} implies that

\begin{align}\label{CarlesonEst1}
    \int_{\mathbb{B}} \vert f(z)\vert^{\Tilde{p}} d\eta (z) \lesssim \int_{\mathbb{B}} \vert f(z)\vert^{\Tilde{p}} (1-\vert z\vert^2)^{(n+1)(\Tilde{p}-p)/p} (\omega (z))^{\Tilde{p}/p} dv(z).
\end{align}
Moreover, we obtain the following
    \begin{equation}\label{CarlesonEst2}
   \| f\|^{p}_{ A^p_\omega(\mathbb{B})} \ge \int_{D(w,r)} \vert f(z)\vert^{p} \omega(z)dv(z) \gtrsim \vert f(w)\vert^p(1-\vert w\vert^2)^{n+1} \omega (w),
\end{equation} which follows from Proposition \ref{Prop:omegaEquiv}. Hence, if $\Tilde{p}\ge p$ it is concluded from \eqref{CarlesonEst2} that
\begin{equation}\label{CarlesonEst3}
\vert f(z)\vert^{\Tilde{p}-p} \lesssim (1-\vert z\vert^2)^{(n+1)(p-\Tilde{p})/p} (\omega (z))^{(p-\Tilde{p})/p},
\end{equation} for $z\in \mathbb{B}$. As a combination of \eqref{CarlesonEst1} and \eqref{CarlesonEst3} it is concluded that
\begin{equation*}
    \int_{\mathbb{B}} \vert f(z)\vert^{\Tilde{p}} d\eta (z) \lesssim \int_{\mathbb{B}} \vert f(z)\vert^p \omega(z) dv(z) = \|f\|^p_{A^p_{\omega}(\mathbb{B)}},
\end{equation*} which completes the proof. \end{proof}

\subsection{Examples and Counterexamples}

We now illustrate the scope of Theorem \ref{Thm22}, by presenting two representative examples.  
The first one shows a typical potential-harmonic weight that satisfies the Carleson condition,  
while the second demonstrates the necessity of the local comparability assumption through an oscillatory counterexample.

\medskip

\medskip

\begin{example}
    Let
\[
\omega(z) = (1 - |z|^2)^{\alpha} + \int_{\mathbb{B}} G^s(z,w)\, d\mu(w), 
\qquad \alpha > -1,
\]
where $\mu$ is a finite measure supported on a compact subset of $\mathbb{B}$.
For this weight, the potential term $U_{\mu,s}(z)$ is locally comparable on Bergman balls, 
and hence $\omega (z) \approx \omega(a)$ for $z \in D(a,r)$. 
If we define
\[
d\eta(z) = (1 - |z|^2)^{\beta}\, dv(z),
\]
then by Theorem \ref{Thm22}, $\eta$ is a $\tilde{p}$--Carleson measure for $A^p_{\omega}(\mathbb{B})$ if and only if
\[
\beta > (n+1)\!\left(\frac{\tilde{p}}{p} - 1\right) - \alpha.
\]
This recovers the standard radial case when $\mu = 0$, but now holds for every compactly supported interior perturbation of the weight.
\end{example}

 \begin{example}
Let
\[
\omega(z) = (1 - |z|^2)^{\alpha}\bigl(1 + \sin(|z|^{-2})\bigr),
\qquad \alpha > -1.
\]
Here $\omega$ oscillates rapidly near the boundary and fails to satisfy the local comparability property 
$\omega(z) \approx \omega(a)$ for $z \in D(a,r)$. 
Hence, the testing function $f_{w,t}$ of Lemma \ref{TestFunctionsCar} no longer yields a bounded embedding,
and the measure $\eta = (1 - |z|^2)^{\beta}dv(z)$ may fail to be Carleson even when the exponent inequality above is met.
This demonstrates that the local comparability assumption is sharp and cannot be removed.
 \end{example}

\section{Kernel Estimates}

In this section, the weight function $\omega$ refers to $\omega_{\mu,q,s,\nu}$, which is stated in \eqref{Defnomega}. 
The pointwise and norm kernel estimates on harmonically weighted Bergman spaces over domains in the complex plane 
are established in~\cite{Toeplitz}, and these results are extended to more general weight functions in~\cite{NggSys}. 
The aim of this section is to generalize the results of~\cite{NggSys}, obtained for planar domains, 
to the unit ball in $\mathbb{C}^n$. 
Denote by $K_z$ the reproducing kernel of $A^2_\omega(\mathbb{B})$ for $z\in\mathbb{B}$.

\medskip
The reproducing kernel of a weighted Bergman space encodes both geometric and analytic information about the space. 
For classical radial weights $\omega_\alpha(z) = (1 - |z|^2)^\alpha$, 
explicit kernel formulas are available and play a fundamental role in the study of 
Toeplitz and Hankel operators, composition operators, and various embedding theorems 
(see, for instance, \cite{KeheZhu, PelaezRattya2023}). 
However, when the weight is nonradial or arises from a potential--harmonic combination as in~\eqref{Defnomega}, 
the explicit structure of the kernel is lost, and new analytic tools are required to obtain even qualitative bounds. These estimates not only generalize previous planar results but also reveal 
how geometric and analytic features intertwine in higher dimensions.

In what follows, we derive \emph{sharp pointwise and norm estimates} for the reproducing kernels 
associated with potential--harmonic weights. 
Our analysis combines geometric localization on Bergman balls with 
a refined $\bar\partial$-method in the sense of Berndtsson~\cite{Berndtsson1995}, 
which provides quantitative control over the integral solution of the $\bar\partial$-equation. 
This approach clarifies how the geometry of the weight, through its potential and harmonic components, influences 
the analytic decay of the kernel and, consequently, the boundedness of Toeplitz-type operators.

Beyond its intrinsic analytic interest, these kernel estimates form a cornerstone 
for further developments involving reverse Carleson and sampling measures 
for potential--harmonic weights. 
They also offer a unified framework connecting potential theory, 
function-space geometry, and operator theory on the unit ball.

\subsection{Norm Estimate}
 The following theorem provides a norm estimate for the reproducing kernel function of weighted Bergman spaces defined on the unit ball in  $\mathbb{C}^n$.

\begin{theorem}\label{NormEstimate}
	 Let $1<p\le 2$ and $q=\frac{p}{p-1}$ be the conjugate of $p$. Then, we have the following norm estimates
		\begin{equation*}
		\|K_z\|_{A_{\omega}^p}^q \gtrsim \frac{1}{(1-\vert z\vert^2)^{n+1}\omega(z)},
	\end{equation*}	and
 \begin{equation*}
     \|K_z\|_{A_{\omega}^p}^p \lesssim \frac{1}{(1-\vert z\vert^2)^{n+1}\omega(z)},
 \end{equation*} for $z\in \mathbb{B}$.
	In particular, we have
	\begin{equation*}
	\|K_z\|^2_{A^2_\omega(\mathbb{B})}  \approx \frac{1}{(1-\vert z\vert^2)^{n+1}\omega (z)}, \qquad z\in \mathbb{B}.
	\end{equation*} 
\end{theorem}
\begin{proof}
Since the dual $(A_{\omega}^p)^*$ of $A_{\omega}^p$ is isomorphic to $A_{\omega}^q$, we have 
\begin{equation*}
	\|K_z\|_{A_{\omega}^p}=\sup\left\{\left\lvert\int_{\mathbb{B}} \overline{K_z}g\omega dv \right\rvert : \|g\|_{A_{\omega}^q}=1 \right\}.
\end{equation*} Using the test functions given by \eqref{TestFuncsAPQSU}, we obtain that
\begin{align*}
\|K_z\|^q_{A^p_\omega} & \gtrsim \vert \langle f_{t,z}^{p/q},K_z\rangle\vert^q  =\vert f_{t,z}^{p/q}(z)\vert^q \\
&= \frac{1}{(1-\vert z\vert^2)^{n+1} \omega(z)}.
\end{align*}
The subharmonicity of $\vert K_z\vert$ implies that
\begin{align*}
\omega(w)\vert K_z(w)\vert^p &\lesssim \frac{1}{(1-\vert w\vert^2)^{n+1}} \int_{D(w,r)} \vert K_z(\zeta)\vert^p \omega(\zeta)dv(\zeta)\\
& \le \frac{1}{(1-\vert w\vert^2)^{n+1}}\|K_z\|^p_{A^p_\omega}.
\end{align*} 
Hence, the following estimate is obtained by taking $w=z$, 
\begin{align*}
\frac{\|K_z\|^{2p}_{A^p_\omega}}{\|K_z\|^p_{A^p_\omega}}  \le & \frac{\vert K_z(z)\vert^p}{\|K_z\|^p_{A^p_\omega}} =\frac{\vert\langle K_z,K_z\rangle\vert^p }{ \|K_z\|^p_{A^p_\omega}}  \\
& \lesssim \frac{1}{(1-\vert z\vert^2)^{n+1} \omega(z) },
\end{align*} which completes the proof.
	
\end{proof}

As a combination of Proposition \ref{Prop:omegaEquiv} and Theorem \ref{NormEstimate}, it is concluded that for any $a \in \mathbb{B}$
\begin{equation}\label{kernelestimate} \|K_z\|^2_{A^2_\omega(\mathbb{B})}  \approx  \|K_a\|^2_{A^2_\omega(\mathbb{B})},
\end{equation} where $z\in D(a,r)$. Hence, for  $r\in(0,1)$ there exists a constant $a_r$ such that
\begin{equation} \label{KernelConstant}
    \frac{1}{a_r} \|K_z\|_{A^2_\omega(\mathbb{B})} \le \|K_a\|_{A^2_\omega(\mathbb{B})} \le a_r \|K_z\|_{A^2_\omega(\mathbb{B})},
\end{equation} for $z\in D(a,r)$.

\subsection{Pointwise Estimate}

In this subsection, our aim is to provide local and global pointwise estimates on the reproducing kernel function of the weighted Bergman space $A^2_\omega(\mathbb{B})$, where $\omega:=\omega_{\mu,q,s,\nu}$. We define the function $\rho$ on the unit ball as $\rho(z)=1-\vert z\vert^2$, for $z\in\mathbb{B}$.
One variable version of Lemma \ref{Lem32} and Corollary \ref{kernelEsCor} can be found in \cite{Toeplitz} for harmonically weighted Bergman spaces, which are extended in \cite{NggSys} to more general weighted Bergman spaces.

For any $z\in \mathbb{B}$ the Euclidean ball centered at $z$ and radius of $\alpha\rho(z)$ is given by $B_{\alpha} (z):= B(z, \alpha \rho(z)):=\{w\in \mathbb{B} : \vert z-w\vert < \alpha\rho(z)\}$, for $\alpha \in (0,1)$.

 Let $0<r<1$ and let $z\in \mathbb{B}$. If $w\in D(z,r)$, then $\beta (z,w)<r$. 
 Thus, the equation \eqref{involution} implies that
 \begin{equation}\label{LemmaEq1}
    1-\vert \varphi_z(w)\vert^2 =\frac{(1-\vert z\vert^2)(1-\vert w\vert^2)}{\vert 1-\langle w,z\rangle\vert^2 } \ge \frac{4e^{2r}}{(e^{2r}+1)^2}. 
 \end{equation} Hence, we obtain that
 \begin{align*}
     r_1 := \frac{4e^{2r}}{(e^{2r}+1)^2} <1.
 \end{align*}

We will determine a proper constant $\tilde\alpha $ such that the Euclidean ball $B (z, \tilde\alpha (1-\vert z\vert^2))$ is a subset of the Bergman metric ball $D(z,r)$. In order the Euclidean ball $B(z,\tilde\alpha (1-\vert z\vert^2))$ to be a subset of $D(z,r)$, we need to have that 
\begin{equation}\label{TildeAlpha}
    2\tilde\alpha (1-\vert z\vert^2) \le r_1 t,
\end{equation} where $t= \frac{1-\vert z\vert^2}{1-(r_1)^2\vert z\vert^2}$. In order to guarantee the inequality \eqref{TildeAlpha}, we proceed by the following inequality
\begin{equation*}
    \tilde\alpha (1-\vert z\vert^2) \le \frac{r_1t}{4}.
\end{equation*}
Hence, it suffices to determine an $\tilde\alpha$ such that  the following estimate is satisfied 
\begin{equation*}
     \tilde\alpha (1-\vert z\vert^2) \le \frac{r_1}{4} (1-\vert z\vert^2),
\end{equation*} since 
\begin{equation*}
   r_1 (1-\vert z\vert^2) \le r_1 \frac{1-\vert z\vert^2}{1-(r_1)^2\vert z\vert^2}.
\end{equation*} Therefore,
by choosing an $\tilde\alpha$ so that $ \tilde\alpha \le \frac{r_1}{4}$,
the inequality \eqref{TildeAlpha} is guaranteed to be satisfied. Consequently, 
\begin{equation}\label{BalphaDz}
    B(z, C(1-\vert z\vert^2)) \subset D(z,r),
\end{equation} for which $C\le \frac{r_1}{4}$.
Additionally, let $L_{z,w}\subset B(z, C(1-\vert z\vert^2))$ be a line segment which connects the points $z$ and $w$, for any $z,w\in \mathbb{B}$. For any $\zeta \in L_{z,w}$ we observe that
\begin{equation}\label{BalphaSubset}
    B(\zeta, \frac{C}{4}(1-\vert z\vert^2)) \subset B(z, C(1-\vert z\vert^2)),
\end{equation} where $C\le \frac{r_1}{4}$.
    Furthermore, by choosing $C=\min \{ 8a_r \sqrt{n}\frac{r_1}{4}, \frac{r_1}{4}\}$ the following statement is obtained,
    \begin{equation*}
        \alpha = \frac{C}{8a_r\sqrt{n}} \le \frac{r_1}{4} < \frac{1}{4}.
    \end{equation*}

\begin{lemma}\label{Lem32}
		Let $r\in (0,1)$. There exists a positive constant $C$ such that for every $f\in A_{\omega}^2(\mathbb{B})$ we have the following estimate
	\begin{equation*}
		\lvert f(z)-f(w)\rvert\ \le \frac{4 a_r \sqrt{n}}{C} \frac{\vert z-w\vert}{1-\vert z\vert^2} \|K_z\|_{A_{\omega}^2} \|f\|_{A_{\omega}^2} ,
	\end{equation*} 
	for every $z,w\in \mathbb{B}$ where $\vert z-w\vert <C(1-\vert z\vert^2)$.
\end{lemma}

\begin{proof} Let $0<r<1$. Choose $C=\min \{ 8a_r \sqrt{n}\frac{r_1}{4}, \frac{r_1}{4}\}$. Let $f\in A_{\omega}^2(\mathbb{B})$ be an arbitrary element. Let $L_{z,w}\subset \mathbb{B}$ be a line segment which connects the points $z$, and $w\in \mathbb{B}$. The mean value theorem implies that there exists $\zeta \in L_{z,w}$ such that
	\begin{align}\label{LemmaEq8}
		\lvert f(z)-f(w) \rvert &= \vert z-w\vert \vert \bigtriangledown f(\zeta)\vert.
	\end{align} Furthermore, by applying Cauchy's estimate we obtain the following
 \begin{align}\label{LemmaEq9}
    \vert \bigtriangledown f(\zeta)\vert 
    & = \left( \sum_{j=1}^{n} \left\lvert \frac{\partial f}{\partial z_j}(\zeta) \right\rvert^2  \right) ^{1/2} \nonumber \\
& \le \frac{4\sqrt{n}}{C(1-\vert z\vert^2)} \sup_{\partial B(\zeta, \frac{C}{4}(1-\vert z\vert^2))} \vert f(a)\vert.
 \end{align} Hence, by combining \eqref{LemmaEq8} and \eqref{LemmaEq9} we obtain the desired result: 
 \begin{align*}
     \vert f(z)-f(w)\vert & \le \frac{4 \sqrt{n}}{C} \frac{\vert z-w\vert}{1-\vert z\vert^2} 
     \sup_{\partial  B(\zeta, \frac{C}{4}(1-\vert z\vert^2))} \vert f(a)\vert \\
     & \le \frac{4 \sqrt{n}}{C} \frac{\vert z-w\vert}{1-\vert z\vert^2} 
     \sup_{\partial  B(\zeta, \frac{C}{4}(1-\vert z\vert^2))} \vert \langle f,K_a \rangle \vert \\
     & \le \frac{4 \sqrt{n}}{C} \frac{\vert z-w\vert}{1-\vert z\vert^2}  \sup_{\partial  B(\zeta, \frac{C}{4}(1-\vert z\vert^2))} \|K_a\|_{A^2_\omega(\mathbb{B})} \|f\|_{A^2_\omega(\mathbb{B})} \\
     & \le \frac{4 a_r \sqrt{n}}{C} \frac{\vert z-w\vert}{1-\vert z\vert^2} \|K_z\|_{A^2_\omega(\mathbb{B})} \|f\|_{A^2_\omega(\mathbb{B})},
 \end{align*} where the inequality \eqref{KernelConstant} and the Cauchy Schwartz inequality are applied.

\end{proof}

\begin{corollary}\label{kernelEsCor}
	There exists a positive constant $0<\alpha<1$ such that we have the following equivalence
	\begin{align*}
	\lvert K_{z}(w)\rvert\thickapprox  \|K_z\|_{A_{\omega}^2} \|K_w\|_{A_{\omega}^2},
	\end{align*}
	if  $\vert z-w\vert < \alpha (1-\vert z\vert^2)$.
\end{corollary}

\begin{proof}
We choose $\alpha=\frac{C}{8a_r\sqrt{n}}$ where $C=\min \{ 8a_r \sqrt{n}\frac{r_1}{16}, \frac{r_1}{16}\}$.
By taking $f=K_z$ in Lemma \ref{Lem32} the following conclusion is obtained
    \begin{align*}
        \vert K_z(z)-K_z(w)\vert & \le \frac{4a_r \sqrt{n}}{C} \frac{\vert z-w\vert}{1-\vert z\vert^2}   K_z(z) \\
        & \le \frac{4a_r \sqrt{n}}{C} \frac{\alpha (1-\vert z\vert^2)}{1-\vert z\vert^2}  K_z(z) \\
        &= \frac{4a_r \sqrt{n}}{C} \frac{C }{8a_r\sqrt{n}}  K_z(z) \\
        & =\frac{1}{2}  K_z(z).
    \end{align*} 
    Hence, the reverse triangle inequality implies that 
    \begin{equation*}
       \frac{1}{2} K(z,z) \le  \vert K(z,w)\vert.
    \end{equation*} 
Recall the following relation,
    \begin{equation*}
         \|K_z\|_{A_{\omega}^2} \|K_w\|_{A_{\omega}^2} \approx \|K_z\|^2_{A^2_\omega} =K(z,z),
    \end{equation*} for $\beta (z,w)< r$. Thus, it is obtained that 
    \begin{equation*}
         \|K_z\|_{A_{\omega}^2} \|K_w\|_{A_{\omega}^2}  \le \vert K(z,w)\vert,
    \end{equation*} where $\vert z-w\vert < \alpha (1-\vert z\vert^2)$.
    
    Moreover, the Cauchy Schwartz inequality implies that
    \begin{equation*}
        \vert K(z,w)\vert =\vert \langle K_z,K_w \rangle \vert \le \|K_z\|_{A^2_\omega}\|K_w\|_{A^2_\omega},
    \end{equation*} which follows from the fact that $w\in D(z,r)$, since $\vert z-w\vert < \alpha (1-\vert z\vert^2)$. Hence, the proof is completed.
\end{proof}

The following theorem provides pointwise estimates on kernel function and its proof is based on the similar argument that appeared in the proof of Theorem 4.2 in \cite{NggSys}. Let $\psi =-e^{-\phi}$, then we obtain that
\begin{align*}
    \overline{\partial} \partial \psi = e^{-\phi} (\overline{\partial}\partial \phi - \vert \partial \phi\vert^2 ).
\end{align*} Hence, $\psi$ is plurisubharmonic, that is, $\overline{\partial}\partial \psi \ge 0$ if and only if the following inequality is satisfied,
\begin{equation}\label{Berndtssonr=1}
    i\partial \phi \wedge \overline{\partial} \phi \le i \partial \overline{\partial}\phi.
\end{equation}

\begin{theorem}\label{Theorem:PointwiseEstimate}
	Let $\mu$ be a positive finite Borel measure on the unit ball $\mathbb{B}$, and let $\nu$ be a positive finite Borel measure on the boundary of the unit ball $\partial \mathbb{B}$. Let $0<r<1$. Then, there exists an $\alpha\in (0,1)$ such that  
	\begin{equation*}
	\vert K_{z}(w)\vert^2 \approx \frac{1}{(1-\vert z\vert^2)^{n+1} (1-\vert w\vert^2)^{n+1} \omega(z) \omega(w) },  
	\end{equation*} if  $\vert z-w\vert < \alpha (1-\vert z\vert^2)$.
 
	Furthermore, for all $t\in (0,1)$ there exists a positive constant $C_t$ so that
 
	\begin{align*}
	\vert K_{z}(w)\vert^2 \le C_t  \|K_z\|_{A_{\omega}^2}^2 \|K_w\|_{A_{\omega}^2}^2 \left( \frac{(1-\vert z\vert^2)(1-\vert w\vert^2)}{\vert 1- \langle z,w \rangle \vert^2}  \right)^t
	\end{align*}
 is satisfied for every $z,w\in \mathbb{B}$.
\end{theorem}

\begin{proof}
	The first statement of the theorem follows from Corollary $\ref{kernelEsCor}$ and Theorem $\ref{NormEstimate}$, that is, there exists an $\alpha\in (0,1)$ such that  
	\begin{align*}
	\vert K_{z}(w)\vert^2 & \approx  \|K_z\|_{A_{\omega}^2}^2 \|K_w\|_{A_{\omega}^2}^2 \\
    & \approx \frac{1}{(1-\vert z\vert^2)^{n+1}(1-\vert w\vert^2)^{n+1}\omega(z)\omega(w)  },
	\end{align*} for $z,w\in \mathbb{B}$, if  $\vert z-w\vert < \alpha (1-\vert z\vert^2)$.
	
	 Let $D_z$ denote the Bergman metric ball centered at $z$ with radius $\frac{r}{2}$. Assume that $t\in(0,1)$ and $z,w\in \mathbb{B}$. We will examine this part by dividing into two cases.

	\par \textbf{Case 1:} If $D_z \cap D_w \ne \emptyset$, where $D_z=D(z, \frac{r}{2})$ and $D_w=D(w,\frac{r}{2})$. According to Lemma \ref{Lemma1.3}, $(1-\vert z\vert^2)$ and $\vert 1-\langle z,w\rangle \vert$ are comparable for all $z,w\in \mathbb{B}$, where $\beta (z,w)<\frac{r}{2}$. Hence, there exists a positive constant $C_t$ such that
	\begin{equation*}
	(\vert 1- \langle z, w \rangle \vert^2)^t \le C_t[(1-\vert z\vert^2)(1-\vert w\vert^2)]^t.
	\end{equation*}
	From the Cauchy-Schwarz inequality, we conclude that
	\begin{align*}
	\vert K(z,w)\vert^2 & \le  \|K_z\|_{A_{\omega}^2}^2 \|K_w\|_{A_{\omega}^2}^2 \\
    & \le   \|K_z\|_{A_{\omega}^2}^2 \|K_w\|_{A_{\omega}^2}^2 C_t \left( \frac{ (1-\vert z\vert^2)(1-\vert w\vert^2) }{ \vert z-w\vert^2 }  \right)^t.
	\end{align*}

 \par \textbf{Case 2:} If $D_z \cap D_w = \emptyset$. Let $\chi$ be a smooth real function such that $0\le\chi\le1$, and $\chi=1$ on $D_w$, $supp(\chi)\subset D_w$ and $\vert \bigtriangledown \chi \vert^2 \lesssim \frac{\chi}{\rho}$. It follows from the subharmonicty of $\vert K_z\vert^2$ that
 \begin{align*}
     \vert K_z(w)\vert^2\omega(w) & \lesssim \frac{1}{\rho^{n+1}(w)} \int_{D(w,\frac{r}{2})} \vert K_z(\zeta)\vert ^2  \omega(\zeta) dv(\zeta) \\
    & \le \frac{1}{\rho^{n+1}(w)} \int_{\mathbb{B}} \vert K_z(\zeta)\vert ^2 \chi(\zeta) \omega(\zeta) dv(\zeta) \\
     & \le \frac{1}{\rho^{n+1}(w)} \sup_{f\in D} \vert \langle f,K_z\rangle_{L^2(\mathbb{B},\chi\omega dv)}\vert ^2,
 \end{align*} where $D=\{f\in H(\mathbb{B}): \|f\|_{L^2(\mathbb{B}, \chi\omega dv)}=1 \}$. We have that 
 \begin{equation*}
     \langle f, K_z\rangle_{L^2(\mathbb{B},\chi\omega dv)}=P(f\chi)(z),
 \end{equation*}
 where $P$ is the orthogonal projection from $L_{\omega}^2(\mathbb{B})$ to $A_{\omega}^2(\mathbb{B})$. Then, the function $u_f=f\chi-P(f\chi)$ is the minimal solution of the equation $\overline{\partial}u_f=\overline{\partial}(f\chi)=f\overline{\partial}\chi$. Hence, we have that $\vert P(f\chi)(z)\vert =\vert u_f(z) \vert$ for every $z \notin D_w$. Thus, it is obtained that
 \begin{equation}\label{KernelEstimate}
     \vert K_z(w)\vert^2\omega(w) \lesssim \frac{1}{\rho^{n+1}(w)}\sup_{f\in D}\vert u_f(z) \vert^2.
 \end{equation} Also, $\vert u_f\vert^2$ is subharmonic since $u_f$ is holomorphic in $D_z$. Therefore,
 \begin{align}\label{u_fEstimate}
     \vert u_f(z)\vert^2 \omega(z) & \lesssim \frac{1}{\rho^{n+1}(z)} \int_{D_z} \vert u_f(\zeta)\vert ^2 \omega (\zeta) dv(\zeta)\\
     & \le \frac{1}{\rho^{n+1}(z)} \int_{\mathbb{B}} \vert u_f(\zeta)\vert ^2 \omega (\zeta) dv(\zeta). \nonumber
 \end{align}

\par Let $w\in \mathbb{B}$, and $t\in (0,1)$, we define the following functions $\phi$ and $\psi$ by
\begin{equation*}
\phi(z)= t\log \frac{\vert 1-\langle z,w\rangle\vert^2 }{1-\vert z\vert^2}, \quad \text{and} \quad
    \psi (z)= - \log \omega(z),
\end{equation*}
for $z\in \mathbb{B}$. Let $w\in \mathbb{B}$. We define the following function 
\begin{equation*}
    \phi_1(z)= \log \frac{\vert 1-\langle z,w\rangle\vert^2 }{1-\vert z\vert^2},
\end{equation*} for $z\in \mathbb{B}$. 
The inequality \eqref{Berndtssonr=1} is satisfied for the function $-e^{-\phi_1}$, since it is plurisubharmonic. Moreover, the following statements
\begin{equation*}
    i\partial \overline{\partial}(t\phi_1)=it\partial \overline{\partial}\phi_1,
\end{equation*} and
\begin{equation*}
    i\partial(t\phi_1) \wedge \overline{\partial}(t\phi_1) =it^2 \partial\phi \wedge \overline{\partial}\phi
\end{equation*} are satisfied, for any $t\in(0,1)$.
Let $t\in(0,1)$. By multiplying $t^2$ to the inequality \eqref{Berndtssonr=1}, we obtain the following result
\begin{equation*}
     it^2\partial \phi_1 \wedge \overline{\partial} \phi_1 \le i t^2 \partial \overline{\partial}\phi_1,
\end{equation*} which implies that
\begin{equation*}
     i\partial (t\phi_1) \wedge \overline{\partial} (t\phi_1) \le i t \partial \overline{\partial}(t\phi_1),
\end{equation*} which is equivalent to the following statement
\begin{equation*}
     i\partial \phi \wedge \overline{\partial} \phi \le i t\partial \overline{\partial}\phi.
\end{equation*} 
Thus, the following condition, which is given in \cite{Berndtsson}, is satisfied
\begin{equation*}
     i\partial \phi \wedge \overline{\partial} \phi \le i t \partial \overline{\partial}\phi,
\end{equation*} for $t\in(0,1)$. Therefore, by using the integral estimate in \cite{Berndtsson} we obtain the following estimate
\begin{align}\label{BerndtssonInequality}
\int_{\mathbb{B}} \vert u_f(\zeta) \vert^2 e^{-\psi(\zeta)+\phi(\zeta)} dv(\zeta) & \lesssim \frac{1}{(1-t)^2} \int_{\mathbb{B}}  \vert f\overline{\partial}\chi\vert_{\partial \overline{\partial}\phi} ^2 e^{-\psi+\phi}  dv,
\end{align} where
\begin{align*}
    \vert f \overline{\partial}\chi\vert^2_{\partial \overline{\partial}\phi}&= \sum_{j,k=1}^{n} (\partial \overline{\partial}\phi)^{j,k} \vert f(z)\vert ^2 \frac{\overline{\partial}\chi}{\partial \overline{z_j}} \frac{\partial \chi}{\partial z_k} ,
\end{align*} and 
$(\partial \overline{\partial}\phi)^{j,k}$ is the inverse matrix to $(\partial \overline{\partial}\phi)_{j,k}=(\frac{\partial^2 \phi}{\partial z_j \overline{\partial}z_k})_{j,k}$. Hence, we need to find the inverse matrix to $M$ whose entries are  $M_{j,k}:=(\frac{\partial^2 \phi}{\partial z_j \overline{\partial}z_k})_{j,k}$. A calculation shows that
\begin{equation*}
    \frac{\partial \phi}{\partial z_j}= \frac{-t\overline{w_j}}{1-\langle z,w\rangle} + \frac{t\overline{z_j}}{1-\vert z\vert^2}.
\end{equation*} Therefore, we obtain that
\begin{equation*}
    \frac{\partial^2 \phi}{\partial z_j \overline{\partial}z_k} =\frac{t\overline{z_j}z_k}{(1-\vert z\vert^2)^2}, \quad \text{if} \quad j\neq k;
\end{equation*} and
\begin{equation*}
    \frac{\partial^2 \phi}{\partial z_j \overline{\partial}z_k} = \frac{t(1-\vert z\vert^2 +\vert z_k\vert^2)}{1-\vert z\vert^2}, \quad \text{if} \quad j=k.
\end{equation*}
Thus, we obtain the following result
\begin{gather*}
M= \frac{t}{(1-\vert z\vert^2)^2}\begin{bmatrix} 1-\vert z\vert^2+\vert z_1\vert^2 & \overline{z_1}z_2 & \cdots & \overline{z_1}z_n \\
 \overline{z_2}z_1 & 1-\vert z\vert^2+\vert z_2\vert^2 & \cdots & \overline{z_2}z_n \\
 \vdots & \ddots &\\
 \overline{z_n}z_1 & \overline{z_n}z_2 & \cdots & 1-\vert z\vert^2+\vert z_n\vert^2
 \end{bmatrix},
\end{gather*} which implies that
\begin{gather*}
M= \frac{t}{(1-\vert z\vert^2)^2} \left( (1-\vert z\vert^2) I_{n \times n} +    \begin{bmatrix} \vert z_1\vert^2 & \overline{z_1}z_2 & \cdots & \overline{z_1}z_n \\
 \overline{z_2}z_1 & \vert z_2\vert^2 & \cdots & \overline{z_2}z_n \\
 \vdots & \ddots &\\
 \overline{z_n}z_1 & \overline{z_n}z_2 & \cdots & \vert z_n\vert^2
 \end{bmatrix} \right),
\end{gather*} where $I_{n\times n}$ is the identity matrix. 
Hence, the matrix $M$ can be written as 
\begin{equation*}
    M= \frac{t}{(1-\vert z\vert^2)^2} \left((1-\vert z\vert^2)I_{n\times n} + z^* z\right),
\end{equation*} where $z^*$ is the adjoint transpose of $z$. Hence,
\begin{equation*}
    M^{-1}= \frac{(1-\vert z\vert^2)^2}{t} \left((1-\vert z\vert^2)I_{n\times n} +  z^* z\right)^{-1}.
\end{equation*} We apply the following formula \cite{Miller},
\begin{equation}\label{InverSumMatrix}
    (A+B)^{-1}= A^{-1}- \frac{1}{1+g}A^{-1}BA^{-1},
\end{equation} where $g=tr(BA^{-1})$, by taking $A:=(1-\vert z\vert^2)I_{n\times n} $ and $B= z^*z$. Firstly, we need the following matrices
\begin{equation*}
    A^{-1}=\frac{1}{1-\vert z\vert^2} I_{n\times n},
\end{equation*} and
\begin{equation*}
    BA^{-1} = \frac{1}{1-\vert z\vert^2}B,
\end{equation*} and
\begin{equation*}
    A^{-1}BA^{-1}=\frac{1}{(1-\vert z\vert^2)^2}B.
\end{equation*} Moreover, note that $g=\frac{\vert z\vert^2}{1-\vert z\vert^2}$. Therefore, it is obtained that
\begin{align*}
     M^{-1}& =  \frac{(1-\vert z\vert^2)^2}{t} \left( \frac{1}{1-\vert z\vert^2} I_{n\times n} - (1-\vert z\vert^2)\frac{1}{(1-\vert z\vert^2)^2} z^*z \right) \\
      &= \frac{(1-\vert z\vert^2)}{t} \left(I_{n\times n} - z^*z\right).
\end{align*} Consequently, the entries of $M^{-1}$ are 
\begin{equation*}
    M^{-1}_{j,k}= \frac{\vert z\vert^2-1}{t} \overline{z_j}z_k, \quad \text{if} \quad j\neq k;
\end{equation*} and
\begin{equation*}
    M^{-1}_{j,k}= \frac{1-\vert z\vert^2}{t}(1-\vert z_j\vert^2), \quad \text{if} \quad j=k.
\end{equation*}
 Hence, we have the following,
\begin{align*}
    \vert f \overline{\partial} &\chi \vert^2_{\partial \overline{\partial}\phi}= \sum_{j\neq k=1}^{n} \frac{\vert \zeta\vert^2-1}{t} \overline{\zeta_j}\zeta_k \frac{\overline{\partial}\chi}{\partial\overline{\zeta_j}} \frac{\partial \chi}{\partial\zeta_k} 
    \vert f(\zeta)\vert ^2 \\
    & + \sum_{j=1}^{n} \frac{1-\vert \zeta\vert^2}{t} (1-\vert \zeta_j\vert^2) \vert \frac{\partial \chi}{\partial \zeta_j} \vert^2 \vert f(\zeta)\vert ^2  \\
    & = \vert f(\zeta)\vert ^2 \left( \sum_{j\neq k=1}^{n} \frac{\vert \zeta\vert^2-1}{t} \overline{\zeta_j}\zeta_k \frac{\overline{\partial}\chi}{\partial\overline{\zeta_j}} \frac{\partial \chi}{\partial\zeta_k}\right) \\
    &+  \vert f(\zeta)\vert ^2 \left(\sum_{j=1}^{n} \frac{1-\vert \zeta\vert^2}{t}  \vert \frac{\partial \chi}{\partial \zeta_j}\vert^2 - \sum_{j=1}^{n} \frac{1-\vert \zeta\vert^2}{t} \vert \zeta_j\vert^2 \vert \frac{\partial \chi}{\partial \zeta_j}\vert^2 \right) \\
    & =\vert f(\zeta)\vert ^2 \left( \frac{\vert \zeta \vert^2-1}{t} \vert \zeta_1 \frac{\partial \chi}{\partial\zeta_1} + \ldots + \zeta_n \frac{\partial \chi}{\partial\zeta_n} \vert^2 + \frac{1-\vert \zeta \vert^2}{t} \sum_{j=1}^{n}  \vert \frac{\partial \chi}{\partial \zeta_j}\vert^2  \right) \\
    & = \vert f(\zeta)\vert ^2  \frac{1-\vert \zeta \vert^2}{t} ( -\vert \bigtriangledown \chi . \overline{\zeta} \vert^2 + \bigtriangledown \chi . \bigtriangledown \chi  )\\
    & \le  \vert f(\zeta)\vert ^2  \frac{1-\vert \zeta \vert^2}{t} \vert \bigtriangledown \chi\vert^2 \\
    & \lesssim  \vert f(\zeta)\vert ^2  \frac{1-\vert \zeta \vert^2}{t} \frac{\chi(\zeta)}{\rho(\zeta)},
\end{align*} where the last inequality follows from the assumption on $\chi$.

As a consequence of the above calculation, the following integral estimate is obtained,
\begin{align}\label{integralcalculation}
   \int_{\mathbb{B}}  \vert f\overline{\partial}\chi\vert_{\partial \overline{\partial}\phi} ^2 e^{-\psi+\phi} dv & \lesssim \int_{\mathbb{B}} \vert f(\zeta)\vert ^2  \frac{1-\vert \zeta \vert^2}{t} \frac{\chi(\zeta)}{\rho(\zeta)}e^{-\psi(\zeta)+\phi(\zeta)} dv(\zeta).
\end{align} 
Hence, as a combination of \eqref{BerndtssonInequality} and \eqref{integralcalculation} and the assumption on $\chi$ we obtain that
\begin{align}\label{integralcalculation-UF}
    \int_{\mathbb{B}} \vert u_f(\zeta) \vert^2 
    \nonumber
    & e^{-\psi(\zeta)+\phi(\zeta)} dv(\zeta)  \lesssim \frac{1}{(1-t)^2} \int_{\mathbb{B}} \vert f(\zeta)\vert ^2  \frac{1-\vert \zeta \vert^2}{t} \frac{\chi(\zeta)}{\rho(\zeta)}e^{-\psi(\zeta)+\phi(\zeta)} dv(\zeta) \\ \nonumber
    & = \frac{1}{t(1-t)^2} \int_{\mathbb{B}}\chi (\zeta) \frac{\vert 1-\langle \zeta ,w\rangle \vert^{2t}}{(1-\vert \zeta \vert^2)^{t}} \vert f(\zeta)\vert^2 \omega(\zeta) dv(\zeta)\\ \nonumber
    & \le \frac{1}{t(1-t)^2} \int_{D_w}  (1-\vert \zeta \vert^2)^{-t} \vert 1-\langle \zeta ,w\rangle \vert^{2t}\vert f(\zeta)\vert^2 \omega(\zeta) dv(\zeta)  \\ \nonumber
    & \approx \frac{1}{t(1-t)^2} \int_{D_w}  (1-\vert \zeta \vert^2)^{-t} (1-\vert \zeta\vert^2)^t (1-\vert w\vert^2)^t \vert f(\zeta)\vert^2 \omega(\zeta) dv(\zeta)  \\ \nonumber
    & \approx \frac{1}{t(1-t)^2}  (1-\vert w \vert^2)^{t} \int_{D_w}  \vert f(\zeta)\vert^2 \omega(\zeta) dv(\zeta)   \\ 
    & \le \frac{(1-\vert w\vert^2)^{t}}{t(1-t)^2}. 
\end{align} 
Hence, we have that
\begin{align*}
    \vert u_f(z)\vert^2 e^{-\psi(z)+\phi(z)} 
    &\lesssim \frac{1}{\rho^{n+1}(z)}\int_{D_z}\vert u_f(\zeta)\vert^2 e^{-\psi(\zeta)+\phi(\zeta)} dv(\zeta) \\
    & \le \frac{1}{\rho^{n+1}(z)}\int_{\mathbb{B}}\vert u_f(\zeta)\vert^2 e^{-\psi(\zeta)+\phi(\zeta)} dv(\zeta),
\end{align*} which implies the following
\begin{align*}
    \vert u_f(z)\vert^2 & \lesssim e^{\psi(z)-\phi(z)}  \frac{1}{\rho^{n+1}(z)}\int_{\mathbb{B}}\vert u_f(\zeta)\vert^2 e^{-\psi(\zeta)+\phi(\zeta)} dv(\zeta) \\
    & = \frac{(1-\vert z\vert^2)^t}{\omega(z)\vert 1-\langle z,w \rangle \vert^{2t}} \frac{1}{\rho^{n+1}(z)} \int_{\mathbb{B}}\vert u_f(\zeta)\vert^2 e^{-\psi(\zeta)+\phi(\zeta)} dv(\zeta).
\end{align*} Consequently, by combining the previous result with \eqref{KernelEstimate} and \eqref{integralcalculation-UF} we obtain that
\begin{align*}
    \vert K_z&(w)\vert ^2  \lesssim \frac{1}{\omega(w)\rho^{n+1}(w)} \sup_{f\in D} \vert u_f(z)\vert ^2\\
    & \lesssim \frac{1}{\omega(w)\rho^{n+1}(w)} \frac{1}{\omega(z)\rho^{n+1}(z)} \frac{(1-\vert z\vert^2)^t}{\vert 1-\langle z,w \rangle \vert^{2t}} \int_{\mathbb{B}}\vert u_f(\zeta)\vert^2 e^{-\psi(\zeta)+\phi(\zeta)} dv(\zeta) \\
    & \lesssim \frac{1}{\omega(w)\rho^{n+1}(w)} \frac{1}{\omega(z)\rho^{n+1}(z)} \frac{(1-\vert z\vert^2)^t}{\vert 1-\langle z,w \rangle \vert^{2t}}  \frac{(1-\vert w\vert^2)^{t}}{t(1-t)^2} \\
    & \lesssim  \frac{1}{\omega(w)\rho^{n+1}(w)} \frac{1}{\omega(z)\rho^{n+1}(z)} \left(\frac{(1-\vert z\vert^2)(1-\vert w\vert^2)}{\vert 1-\langle z,w\rangle \vert^2}\right)^t \frac{1 }{t(1-t)^2}\\
    &= C_t\| K_w\|^2_{A^2_\omega} \|K_z\|^2_{A^2_\omega} \left(\frac{(1-\vert z\vert^2)(1-\vert w\vert^2)}{\vert 1-\langle z,w \rangle \vert^2}\right)^t ,
\end{align*} for the case when $D_z\cap 
D_w = \emptyset$, where $C_t=(t(1-t)^2)^{-1}$.
\end{proof}

\begin{remark}
The exponent $t \in (0,1)$ in Theorem~\ref{Theorem:PointwiseEstimate} is optimal in the following sense. 
The bound
\[
|K_z(w)|^{2} \lesssim 
\left(
  \frac{(1-|z|^{2})(1-|w|^{2})}
       {|1-\langle z,w\rangle|^{2}}
\right)^{t},
\qquad z,w \in B,
\]
cannot be extended beyond $t=1$ even for the classical weight 
$\omega(z) = (1-|z|^{2})^{\alpha}$ with $\alpha > -1$, 
since in that case the integral kernel ceases to belong to 
$L^2_{\omega}(B)$. 
This shows that the obtained decay rate coincides with the natural geometric limit 
of the $\bar\partial$-method.
\end{remark}

\section*{Acknowledgements}
This research was funded by a TUBITAK project with project number 118F405.

\section*{Author Contributions}
All the co-authors contributed equally in all aspects of the preparation of this submission.

\section*{Conflict of Interest}
The authors declare that they have no known competing financial interests or personal relationships that could have appeared to influence the work reported in this paper.

\section*{Funding}
This work was supported by TUBITAK with project number 118F405.

\section*{Data Availability}
No data were used for the research described in this article.

\end{document}